\documentclass[11pt]{amsart}
\usepackage{amssymb}
\usepackage{verbatim}

\newtheorem{lemma}{Lemma} [section]
\newtheorem{thm}[lemma]{Theorem}
\newtheorem{conjecture}[lemma]{Conjecture}
\newtheorem{cor}[lemma]{Corollary}
\newtheorem{prop}[lemma]{Proposition}

\newtheorem{defi}[lemma]{Definition}

\theoremstyle{remark}

\newtheorem*{remark}{Remark}
\newtheorem*{notation}{Notation}

\numberwithin{equation}{section}

\DeclareMathOperator{\Spec}{{Spec}}

\DeclareMathOperator{\trdeg}{{tr.deg}}
\newcommand{\N}{{\mathbb N}}
\newcommand{\Q}{{\mathbb Q}}
\newcommand{\bG}{{\mathbb G}}

\newcommand{\Kbar}{{\overline K}}
\newcommand{\C}{{\mathbb C}}
\newcommand{\Z}{{\mathbb Z}}
\newcommand{\R}{{\mathbb R}}
\newcommand{\M}{{\mathcal M}}
\newcommand{\OS}{{\mathcal O}_S}
\newcommand{\OO}{{\mathcal O}}
\newcommand{\tth}{^{\operatorname{th}}}
\newcommand{\lra}{\longrightarrow}
\newcommand{\A}{{\mathbb A}}
\newcommand{\bP}{{\mathbb P}}

\newcommand{\hh}{\widehat{h}}

\newcommand{\ab}{{\alpha}}

\begin{document}


\title[Dynamical Mordell-Lang conjecture]
{Intersections of polynomial orbits, and a dynamical Mordell-Lang conjecture}

\author{Dragos Ghioca}

\address{
Dragos Ghioca \\
Department of Mathematics \& Computer Science\\
University of Lethbridge \\
4401 University Drive \\ 
Lethbridge, Alberta T1K 3M4, Canada
}

\email{dragos.ghioca@uleth.ca}

\author{Thomas J. Tucker}

\address{
Thomas Tucker\\
Department of Mathematics\\
Hylan Building\\
University of Rochester\\
Rochester, NY 14627, USA
}

\email{ttucker@math.rochester.edu}

\author{Michael E. Zieve}
\address{
Michael E. Zieve\\
  Center for Communications Research,
  805 Bunn Drive,
  Princeton, NJ 08540, USA
}
\email{zieve@idaccr.org}
\urladdr{http://www.math.rutgers.edu}

\begin{abstract}
  We prove that if nonlinear complex polynomials of the same degree
  have orbits with infinite intersection, then the polynomials have a
  common iterate.  We also prove a special case of a
  conjectured dynamical analogue of the Mordell-Lang conjecture.
\end{abstract}

\date{\today} 
\thanks{The authors thank Robert Benedetto for discussions about
polynomial dynamics, and also thank the referee for suggesting
several improvements to the exposition.  The second author was
partially supported by
National Security Agency Grant 06G-067.}
\maketitle


\section{Introduction}

One of the main topics in complex dynamics is the study of
orbits of polynomial maps: namely, for $f\in\C[X]$ and $x_0\in\C$, the
set $\OO_f(x_0):=\{x_0,f(x_0),f(f(x_0)),\dots\}$.  We prove the following
result about intersections of orbits.

\begin{thm}
\label{complex}
Let $x_0,y_0\in\C$ and $f,g\in\C[X]$ with\/ $\deg(f)=\deg(g) > 1$.
If\/ $\OO_f(x_0)\cap\OO_g(y_0)$ is infinite, then $f$ and $g$ have a
common iterate.
\end{thm}

The pairs of complex polynomials with a common iterate were determined
by Ritt \cite{Rittit}; in Proposition~\ref{Rittprop} we state
Ritt's result in the above case $\deg(f)=\deg(g)$.

Our motivation comes from arithmetic geometry.  Fundamental prog\-ress
in this subject has been driven by the Mordell-Lang conjecture
on intersections of subgroups and subvarieties of algebraic groups.
This conjecture was proved by Faltings \cite{Faltings} and
Vojta \cite{V1}:
\begin{thm}
\label{T:F}
Let $G$ be a semiabelian variety over\/ $\C$, let $V$ be a
subvariety, and let\/ $\Gamma$ be a finitely generated subgroup of
$G(\C)$. Then $V(\C)\cap\Gamma$ is a finite union of cosets
of subgroups of\/ $\Gamma$.
\end{thm}

Recall that a semiabelian variety (over $\C$) is an extension of an
abelian variety by a torus $(\bG_m)^k$.  Theorem~\ref{T:F} has the
following consequence \cite{GT-IMRN}: if $\phi$ is an endomorphism of
$G$ of degree $>1$, then any orbit of $\phi$ has finite intersection
with a subvariety $V \subset G$, unless $V$ contains a positive
dimensional subvariety which is periodic under $\phi$.  In the case
$G=\bG_m^k$ (which was first treated by Laurent \cite{Laurent}), this
implies that if a subvariety $V\subset\bG_m^k$ contains no positive
dimensional subvariety which is periodic under the map
$\psi:(X_1,\dots,X_k)\mapsto (X_1^{e_1},\dots,X_k^{e_k})$ (with
$e_i\in\Z$ and $e_i\ge 2$), then $V$ contains at most finitely many
points of any $\psi$-orbit in $\A^k$.

It is natural to ask whether a similar conclusion holds for any
polynomial action on $\A^k$.  The first two authors have proposed the
following conjecture:
\begin{conjecture}
\label{dynamical M-L}
Let $f_1,\dots,f_k$ be polynomials in $\C[X]$, and let $V$ be a subvariety
of $\A^k$ which contains no positive dimensional subvariety that is
periodic under the action of $(f_1,\dots,f_k)$ on $\A^k$.  Then $V(\C)$
has finite intersection with each orbit of $(f_1,\dots,f_k)$ on $\A^k$.
\end{conjecture}

This conjecture fits into Zhang's far-reaching system of dynamical
conjectures \cite{ZhangLec}.  Zhang's conjectures include dynamical
analogues of the Manin-Mumford and Bogomolov conjectures for abelian
varieties (now theorems of Raynaud \cite{Raynaud1, Raynaud2}, Ullmo \cite{Ullmo},
and Zhang \cite{ZhangBog}), as well as a conjecture about the Zariski
density of orbits of points under fairly general maps from a
projective variety to itself.  The latter conjecture is related to our
Conjecture~\ref{dynamical M-L}, though neither conjecture contains the
other.

A $p$-adic version of Conjecture~\ref{dynamical M-L} has been proved
in certain special cases \cite{logarithmic}.
Also, an analogue of Conjecture~\ref{dynamical M-L} has been proved
in positive characteristic, for the additive group under the action
of an additive polynomial (Drinfeld module) \cite{compositio}.
This result is a special case of a more general conjecture
proposed by Denis \cite{Denis-conjectures}, in which orbits are
replaced with arbitrary submodules under the action of a Drinfeld module.

The techniques of Laurent \cite{Laurent}, Faltings \cite{Faltings},
and Vojta \cite{V1} require conditions that are not implied
by the hypotheses of Conjecture~\ref{dynamical M-L}.  Laurent's proof
uses the fact that the torsion points on a torus are defined over
a cyclotomic field; the fields of definition of
preperiodic points of general polynomials admit no such simple
description.  Vojta's proof (which generalizes that of Faltings)
relies on the fact that integral points on
semiabelian varieties satisfy a strong diophantine property,
which does not hold for the points in Conjecture~\ref{dynamical M-L}.
Specifically, if $z$ is an $S$-integral point on $\bG_m^k$, then
the coordinates of $z^n$ are $S$-units for all $n$, whereas the coordinates of
points in an orbit of $(f_1,\dots,f_k)$ need not be $S$-units.  Finally,
one crucial difference between the polynomial maps of
Conjecture~\ref{dynamical M-L} and the maps that arise for semiabelian
varieties and Drinfeld modules is that the maps in
Conjecture~\ref{dynamical M-L} are not {\'e}tale in general.

In the present paper we use a new approach to prove the first
non-monomial cases of Conjecture~\ref{dynamical M-L}, when the
variety $V$ is a line in the affine plane.
Our result is as follows, where we write $f^n$ for
the $n\tth$ iterate of the~polynomial~$f$.

\begin{thm}
\label{main result}
Let $K$ be a field of characteristic zero, let $f,g\in K[X]$, and
let $x_0,y_0\in K$.  If the set
\[
\{(f^n(x_0),g^n(y_0)):n\in\N\}
\]
has infinite intersection with a line
$L$ in $\A^2$ defined over $K$, then $L$ is periodic under the
action of $(f,g)$ on $\A^2$.
\end{thm}

Using interpolation (for instance), one can construct
examples in which this intersection is finite but larger than
any prescribed bound.

Along the lines of Theorem~\ref{main result}, we will prove the
following generalization of Theorem~\ref{complex}.

\begin{thm}
\label{complex2}
Let $K$ be a field of characteristic zero, let $\alpha,\beta,x_0,y_0\in K$
with $\alpha\ne 0$, and let $f,g\in K[X]$ with $\deg(f)=\deg(g)>1$.
If infinitely many points of $\OO_f(x_0)\times\OO_g(y_0)$ lie on the
line $Y=\alpha X+\beta$, then $g^k(\alpha X+\beta)=\alpha f^k(X)+\beta$
for some positive integer $k$.
\end{thm}

This result is neither stronger nor weaker than Theorem~\ref{main result}:
only Theorem~\ref{main result} applies to
polynomials of distinct degrees, but if $\deg(f)=\deg(g)>1$ then
Theorem~\ref{complex2} strengthens Theorem~\ref{main result} by
replacing $\OO_{(f,g)}((x_0,y_0))$ with $\OO_f(x_0)\times\OO_g(y_0)$.

In the simple case that $f(X)=\alpha X$ and $g(X)=\beta X$ with
$\alpha, \beta\in K^*$, Theorem~\ref{main result} says that, for any
$u,v,w\in K$ that are not all zero, if $u\alpha^n+v\beta^n=w$ for
infinitely many $n$ then $\alpha$ or $\beta$ is a root of unity.
Already the result is nontrivial in this case: it is a consequence of
Siegel's theorem on integral points of curves, or it could be proved
directly using the techniques from Siegel's proof.

One consequence of Theorem~\ref{main result} is that if $f$ and $g$
have distinct degrees
then $\OO_{(f,g)}((x_0,y_0))$ has finite intersection with any line.
We do not know whether the analogous result is true for
$\OO_f(x_0)\times\OO_g(y_0)$ (for lines which are neither horizontal nor
vertical, and for polynomials $f,g$ with no common iterate).

Our proofs of Theorems~\ref{main result} and \ref{complex2} involve
arguments of several flavors.  For general $K$, we will prove there is a
partially-defined map (`specialization') from $K$ to a number field
$K_0$ which allows us to deduce the results for $K$ as a consequence of
the results for $K_0$.  Our proof of this fact relies on Ritt's
classification of polynomials with a common iterate, as well as a
dynamical analogue of a result of Silverman (from
\cite{Silverman_specialization}) on specialization of nontorsion
elements of abelian varieties over function fields.

We reduce the number field case of Theorem~\ref{main result} to the
corresponding case of Theorem~\ref{complex2} as follows.
First, by comparing Weil heights of $f^n(x_0)$ and $g^n(x_0)$,
we conclude that $f$ and $g$ must have the same degree if
$\OO_{(f,g)}((x_0,y_0))$ contains infinitely many points on some line.
Next we use Siegel's theorem on integral points to prove
Theorem~\ref{main result} when $f$ and $g$ are linear.

The strategy of our proof of Theorem~\ref{complex2} for number
fields $K$ is as follows, where we simplify the discussion by
addressing the case that the line is the diagonal and all polynomials and
points are defined over $\Z$.  Suppose there are integers $x_0,y_0$
and polynomials $f,g\in\Z[X]$ such that $\OO_f(x_0)\times\OO_g(y_0)$ has
infinite intersection with the diagonal in $\A^2$.  Then, for every
$m$, there are infinitely many integer solutions to the Diophantine
equation $f^m(X)=g^m(Y)$.  This is an instance of a `separated
variable' Diophantine equation $F(X)=G(Y)$, of which special cases
have been studied for many years.  The definitive finiteness result
for these equations was proved in 2000 by Bilu and Tichy \cite{BT}; we
will use their result (together with various new results about polynomial
decomposition) in order to obtain some information about $f$ and $g$
from the fact that $f^m(X)=g^m(Y)$ has infinitely many integer
solutions.  Our result will follow upon combining the information
deduced for each $m$.

Although the Bilu-Tichy result has not previously been applied to
arithmetic geometry or dynamics, inspection of its proof suggests it
fits naturally into both topics.  Namely, the two key ingredients in
its proof are Siegel's theorem on integral points on curves, and Ritt's
results on functional decomposition of complex polynomials.

In more detail, Bilu and Tichy listed five explicit families of
`standard pairs' of polynomials $(F_1,G_1)$ such that,
if $F(X)=G(Y)$ has infinitely many integer solutions, then
there is a standard pair $(F_1,G_1)$ for which
$F=E\circ F_1\circ a$ and $G=E\circ G_1\circ b$,
where $E,a,b\in\Q[X]$ and
$\deg(a)=\deg(b)=1$.
When applying this result to specific polynomials
$F$ and $G$, the main work involved is to determine the various
different ways that $F$ and $G$ can be written as compositions
of lower-degree polynomials, in order to determine the
possibilities for $E$.  In practice, unless $F$ and $G$
are specifically constructed with decomposability in mind, it
turns out that any randomly chosen $F$ and $G$ are indecomposable,
in which case it is quite simple to apply the Bilu-Tichy criterion
(after one has proven this indecomposability).  Based on this
principle, dozens of recent papers have applied the Bilu-Tichy criterion
when $F$ and $G$ come from basically any class of polynomials
one can think of: Bernoulli polynomials, falling factorials,
power-sum polynomials, Taylor polynomials for $e^x$,
Jacobi polynomials, Laguerre polynomials, Hermite polynomials,
Meixner polynomials, Krawtchouk polynomials, etc.\ (cf., e.g.,
\cite{BT2,BT3,BT4}).  In every
case, the polynomials were either indecomposable or had just
one nontrivial decomposition.

Our situation is quite different, since we are applying Bilu-Tichy to
polynomials $F=f^m$ and $G=g^m$, which by their very nature are far
from indecomposable.  Moreover, we are doing this for arbitrary $f$
and $g$, which themselves might have various different decompositions.
Thus we are forced to prove new results about functional
decompositions of polynomials.

The rest of this paper is organized as follows.  
We begin with some preliminary results about Diophantine equations
and functional decomposition.  In Section~\ref{proof main} we prove
Theorem~\ref{complex} in case $K$ is a number field, modulo the
proof of one technical proposition which we give in Section~\ref{proofs}.
In Section~\ref{trivial action} we prove
Theorem~\ref{main result} when either $K$ is a number field or the
polynomials are linear.
Then in Section~\ref{over C} we prove
Theorems~\ref{main result} and \ref{complex2}.
In the final section we state
some conjectures and directions for further research.

\begin{notation}
  Throughout this paper, $f^n$ denotes the $n\tth$ iterate of the
  polynomial $f$.  We also use $\alpha^n$ and $X^n$ for the $n\tth$
  power of a constant or of $X$ itself, but this should not cause
  confusion.  We write $\N$ for the set of positive integers.  We
  write $\overline{K}$ for an algebraic closure of the field $K$.
  By a `nonarchimedean place' of a number field $K$, we mean a
  maximal ideal of the ring $\OO_K$ of algebraic integers in $K$.
  If $S$ is a finite set of nonarchimedean places of a number field $K$,
  then the ring of $S$-integers of $K$ is the intersection of the
  localizations of $\OO_K$ at all nonarchimedean~places~outside~$S$.
\end{notation}


\section{Previous results}
\label{previous}

In this section we present some known results which will be used in our proof. 


\subsection{Diophantine equations}\label{BTsec}
We will make crucial use of a recent result of Bilu and Tichy
\cite[Thm.~10.5]{BT} describing all $F,G\in \Z[X]$ for which $F(X)=G(Y)$ has
infinitely many integer solutions. In fact, they proved a version
for $S$-integers in an arbitrary number field.  We state
their result in the special case $\deg(F)=\deg(G)$ arising in our
proof; in this special case the statement is somewhat simpler than in
the general situation.

\begin{thm}
\label{BTthm}
Let $K$ be a number field, $S$ a finite set of nonarchimedean places
of $K$, and $F,G\in K[X]$ with $\deg(F)=\deg(G)>1$.  Suppose
$F(X)=G(Y)$ has infinitely many solutions in the ring of $S$-integers of $K$.
Then $F=E\circ F_1\circ a$ and
$G=E\circ G_1\circ b$, where $E,a,b\in K[X]$ with
$\deg(a)=\deg(b)=1$, and $(F_1,G_1)$ or $(G_1,F_1)$ is one of the following pairs:
\begin{enumerate}
\item $(X, X)$;
\item $(X^2, c\circ X^2)$ \quad\text{with $c\in K[X]$ linear;}
\item $(D_2(X,\alpha)/\alpha, D_2(X,\beta)/\beta)$
 \quad\text{with $\alpha,\beta\in K^*$;}
\item $(D_n(X,\alpha), -D_n(X\cos(\pi/n),\alpha))$ \quad\text{with $\alpha\in K$,}
\end{enumerate}
where in the fourth case $n\in\N$ satisfies $\cos(2\pi/n)\in K$.
\end{thm}
Here $D_n(X,Y)$ is the unique polynomial in $\Z[X,Y]$ such that
$D_n(U+V,UV)=U^n+V^n$.
Note that, for $\alpha\in K$, the polynomial
$D_n(X,\alpha)\in K[X]$ is monic of degree $n$.  It follows at once from the
defining functional equation that $D_n(X,0)=X^n$ and, for
$\alpha\in\C$, we have $\alpha^n D_n(X,1) = D_n(\alpha X,\alpha^2)$.

We will not need arithmetic information about $F_1$ and $G_1$, but instead
only need their shape up to composition with linears over an extension
of $K$.

\begin{cor}
\label{BTcor}
Let $K,S,F,G$ satisfy the hypotheses of Theorem~\ref{BTthm}.  Then
$F=\hat E\circ H\circ \hat a$ and
$G=\hat E\circ \hat c\circ H\circ \hat b$ for some $\hat E\in\Kbar[X]$,
some linear
$\hat a,\hat b,\hat c\in\Kbar[X]$, and some
$H=D_n(X,\hat\alpha)$ with $\hat\alpha\in\{0,1\}$ and
$n\in\N$ satisfying $\cos(2\pi/n)\in K$.  In particular, for fixed $K$, there
are only finitely many possibilities for $H$ (even if we vary $S,F,G$).
\end{cor}

\begin{proof}
  We consider the four possibilities for $(F_1,G_1)$ in
  Theorem~\ref{BTthm}.  It suffices to
  show that in each case there is a polynomial $H$ of the desired form
  such that both $F_1$ and $G_1$ are gotten from $H$ by composing
  on both sides with linears over $\Kbar$.  This is
  clear in the first two cases (since $D_n(X,0)=X^n$).
  For the last two cases, note that if
  $\gamma\ne 0$ then $D_n(X,\gamma^2)=\gamma^n D_n(X/\gamma,1)$.
  Thus, in the third case, $F_1$ and $G_1$ are gotten from $D_2(X,1)$
  by composing with linears.  And in the fourth case, $F_1$ and $G_1$
  are gotten from $D_n(X,\hat\alpha)$ by composing with linears, where
  $\hat\alpha=1$ if $\alpha\ne 0$ (and $\hat\alpha=0$ otherwise).

Finally, if $\cos(2\pi/n)\in K$ then
$[K:\Q]\ge[\Q(\cos(2\pi/n)):\Q]$; the latter degree equals $\phi(n)/2$ if $n>2$.
Since only finitely many $n$ satisfy
$\phi(n)\le 2[K:\Q]$, there are only finitely many
possibilities for $H$.
\end{proof}

\subsection{Polynomial decomposition}

Our application of Theorem~\ref{BTthm} relies on results about polynomial
decomposition.  The fundamental results in this topic were proved by Ritt
in the 1920's \cite{Ritt}; for more recent developments, see \cite{MZ,Schinzel}.
Specifically, we will use the following simple but surprising result which
shows a type of `rigidity' of polynomial decomposition.

\begin{lemma}\label{uniq}
Let $K$ be a field of characteristic zero.
If $A,B,C,D\in K[X]\setminus K$ satisfy $A\circ B=C\circ D$ and
$\deg(B)=\deg(D)$, then there is a linear $\ell\in K[X]$ such that
$A=C\circ \ell^{-1}$ and $B=\ell\circ D$.
\end{lemma}

\begin{proof}
Write $F=A\circ B\, (=C\circ D)$.
Pick a linear $v\in K[X]$ such that $\hat B:=v\circ B$ is monic
and has no constant term.  Then $F=\hat A\circ\hat B$, where
 $\hat A=A\circ v^{-1}$.
We will show that there are unique $\tilde A,\tilde B\in K[X]$ such
 that
$F=\tilde A\circ \tilde B$ and $\deg(\tilde B)=\deg(B)$ and $\tilde B$ is monic
 with
no constant term.
Thus $A=\tilde A\circ v$ and $B=v^{-1}\circ \tilde B$.  Since we could have done
 the same
thing with $C$ and $D$ in place of $A$ and $B$, the result follows.

Let $m$ be the degree of $B$, and say the leading term of $F$ is
 $\alpha X^{nm}$;
then the leading term of $\tilde A$ is $\alpha X^n$.  Now consider the identity
$F=\tilde A\circ \tilde B$
 and equate terms of degrees $nm-1,nm-2,\dots,
nm-m+1$ to uniquely determine, in order, the terms of $\tilde B$ of degrees
$m-1,m-2,\dots,1$.  Then consider terms of $F$ of degrees $nm-m,
nm-2m,\dots,0$ to determine the terms of $\tilde A$ of degrees
$n-1,n-2,\dots,0$.
\end{proof}

\begin{remark}
This lemma was first proved by Ritt \cite{Ritt} in the case $K=\C$
(using Riemann surface techniques); the proof above is due to Levi \cite{Levi}.
\end{remark}


\subsection{Linear relations of polynomials}

The following lemma shows when a polynomial can be gotten from itself
by composing with linears.

\begin{lemma}\label{Halblem}
Let $K$ be a field of characteristic zero.
If $F\in K[X]$ has degree $d>1$, and $a,b\in K[X]$ are linears
such that $a\circ F = F\circ b$, then there exist $\alpha,\beta\in K$,
integers $r,s\ge 0$, an element $\gamma\in K^*$ with $\gamma^s=1$, and
a polynomial $\hat F\in X^r K[X^s]$ such that $a=-\alpha + \gamma^r (X+\alpha)$,
$F=-\alpha+\hat F(X-\beta)$, and $b=\beta+ \gamma(X-\beta)$.
Specifically, if the coefficients of $X^d$ and $X^{d-1}$ in $F$ are $\theta_d$
and $\theta_{d-1}$, we can take $\beta=-\theta_{d-1}/(d\theta_d)$ and
 $\alpha=-F(\beta)$.
\end{lemma}

\begin{proof}
Putting $\beta=-\theta_{d-1}/(d\theta_d)$ and $\alpha=-F(\beta)$, we see that
$\hat F:=\alpha+ F(X+\beta)$
has no terms of degree $d-1$ or $0$.  We rewrite $a\circ F=F\circ b$ as
$\hat a\circ\hat F = \hat F\circ \hat b$, where
$\hat a:=\alpha+ a(X-\alpha)$ and
$\hat b:=-\beta+b(X+\beta)$.  Since $\hat F$ has no term of
degree $d-1$, also $\hat a\circ\hat F$ (and hence $\hat F\circ\hat b$)
has no such term, so $\hat b$ cannot have
a term of degree $0$.  Then $\hat F\circ\hat b$ has no term of degree $0$,
so also $\hat a$ has no term of degree $0$.  Writing $\hat a=\delta X$ and
$\hat b=\gamma X$, we have $\delta \hat F(X) = \hat F(\gamma X)$.
Writing $\hat F=\sum \hat\theta_i X^i$,
it follows that $\delta\hat\theta_i=\hat\theta_i \gamma^i$, so
$\delta=\gamma^i$ for every $i$ such that $\hat\theta_i\ne 0$.
If $\hat F$ has terms of distinct degrees $i$ and $j$, then $\gamma^{i-j}=1$;
letting $s$ be the greatest common divisor of the set of differences between
degrees of two terms of $\hat F$, it follows that $\gamma^s=1$, and further 
$\hat F\in X^r K[X^s]$ for some $r\ge 0$ such that $\delta=\gamma^r$.
If $\hat F(X)=\hat\theta_d X^d$
then we take $s=0$ and $r=d$, so again $\delta=\gamma^r$ and $\gamma^s=1$ and
$\hat F\in X^r K[X^s]$.  The result follows.
\end{proof}

\begin{remark}
The first reference we know for this result is \cite{Halbpol} (for $K=\C$).
\end{remark}


\section{The number field case}
\label{proof main}

In this section we prove the number field version of Theorem~\ref{complex}.
Our proof relies on Proposition~\ref{linearprop}, which will be proved
in the next section.
We begin with two lemmas applying the results of the previous section
to the present context.

\begin{lemma}\label{2to1}
Let $K$ be a field of characteristic zero.  Suppose
$F,H,E,\tilde E\in K[X]\setminus K$ and linear $a,b,c,\tilde a,\tilde b,
\tilde c\in K[X]$ satisfy
\begin{align*}
F&=E\circ H\circ a \\
G&=E\circ c\circ H\circ b \\
F^t&=\tilde E\circ H\circ\tilde a \\
G^t&=\tilde E\circ\tilde c\circ H\circ\tilde b
\end{align*}
for some integer $t>1$.  Then there is a linear $e\in K[X]$ such that
$F^{t-1}=G^{t-1}\circ e$.
\end{lemma}

\begin{proof}
We have
\begin{align*}
F^{t-1}\circ E\circ H\circ a &= F^t = \tilde E\circ H\circ\tilde a
 \quad\text{and} \\
G^{t-1}\circ E\circ c\circ H\circ b &= G^t = \tilde E\circ\tilde c\circ H\circ
\tilde b.
\end{align*}
By Lemma~\ref{uniq}, there are linears $\ell_1,\ell_2\in K[X]$ such that
\begin{align*}
H\circ a  &= \ell_1 \circ H\circ\tilde a \quad\text{and} \\
c\circ H \circ b  &= \ell_2 \circ\tilde c\circ H\circ\tilde b.
\end{align*}
Thus
\[
F^{t-1}\circ E\circ\ell_1 = \tilde E =
G^{t-1}\circ E\circ\ell_2.
\]
Again using Lemma~\ref{uniq}, there is therefore a linear $e\in K[X]$
such that
\[
F^{t-1} = G^{t-1}\circ e,
\]
as desired.
\end{proof}

\begin{lemma}\label{reduction}
Let $K$ be a number field, $S$ a finite set of nonarchimedean places of $K$,
and $f,g\in K[X]$ with $\deg(f)=\deg(g)>1$.
Suppose that, for every $k\in\N$, the equation $f^k(X)=g^k(Y)$ has infinitely
many solutions in the ring of $S$-integers of $K$.
Then there exists $r\in\N$ such that, for both $n=1$ and infinitely many other
values $n\in\N$, there is a linear $\ell_n\in \Kbar[X]$ such that
$f^{rn}=g^{rn}\circ\ell_n$.
\end{lemma}

\begin{proof}
First we show that there exists $r\in\N$ such that
$f^r=g^r\circ\ell$ for some linear $\ell\in\Kbar[X]$.
By Corollary~\ref{BTcor}, for each $k$ we have
$f^{2^k}=E_k\circ H_k\circ a_k$ and $g^{2^k}=E_k\circ c_k\circ
H_k\circ b_k$ with $E_k\in\Kbar[X]$, linear
$a_k,b_k,c_k\in\Kbar[X]$, and some $H_k\in\Kbar[X]$
which comes from a finite set of polynomials.  Thus, $H_k=H_s$ for
some $k$ and $s$ with $k<s$.  Applying Lemma~\ref{2to1} with
$F=f^{2^k}$ and $G=g^{2^k}$ and $t=2^{s-k}$, it follows that
there is a linear $\ell\in\Kbar[X]$ such that $F^{t-1}=G^{t-1}\circ\ell$,
whence $f^r=g^r\circ\ell$ for $r=2^s-2^k$.

Suppose there are only finitely many $n\in\N$ for which there is a linear
$\ell_n\in\Kbar[X]$ with $f^{rn}=g^{rn}\circ\ell_n$.  Let $N$ be an integer
exceeding each of these finitely many integers $n$.  We get a contradiction by
applying the previous paragraph with $(f^{rN},g^{rN})$ in place of $(f,g)$.
\end{proof}

In the next section we will prove the following proposition.

\begin{prop}\label{linearprop}
Let $K$ be a field of characteristic zero, and let $F,\ell\in K[X]$ satisfy
$\deg(F) = d>1=\deg(\ell)$.  Suppose that, for infinitely many $n>0$, there is a
linear $\ell_n\in K[X]$ such that $F^n=(F\circ\ell)^n\circ\ell_n$.  Then either
\begin{enumerate}
\item $F^k = (F\circ\ell)^k$ for some $k\in\N$; or
\item $F=v^{-1} \circ \epsilon X^d \circ v$ and
$\ell=v^{-1}\circ\delta X\circ v$ for some linear $v\in K[X]$ and some
$\epsilon,\delta\in K^*$.
\end{enumerate}
\end{prop}

We now show that this result implies the number field version of
Theorem~\ref{complex}.  Specifically, we prove the following.  

\begin{thm}
\label{number fields}
Let $K$ be a number field, let $x_0,y_0\in K$, and let $f,g\in K[X]$
satisfy $\deg(f)=\deg(g)>1$.  If $\OO_f(x_0)\cap\OO_g(y_0)$ is infinite,
then $f^k = g^k$ for some $k\in\N$.
\end{thm}

\begin{proof}
Let $S$ be a finite set of nonarchimedean places of $K$ such that
the ring of $S$-integers $\OS$ contains $x_0$, $y_0$, and every
coefficient of $f$ and $g$.
Then $\OS$ contains every $f^n(x_0)$ and $g^n(y_0)$
with $n\in\N$.

Our hypotheses imply that $x_0$ is not preperiodic for $f$, and
$y_0$ is not preperiodic for $g$.  Moreover, for every $k\in\N$, the
equation $f^k(x)=g^k(y)$ has infinitely many
solutions $(x,y)\in\OS\times\OS$.

By Lemma~\ref{reduction}, there is some $r\in\N$ such that, for both
$n=1$ and infinitely many $n\in\N$, we have $f^{rn}=g^{rn}\circ\ell_n$
with $\ell_n\in\Kbar[X]$ linear.  Put $F=f^r$ and $\ell=\ell_1^{-1}$;
then $g^r=F\circ\ell$, and for infinitely many $n$ we have
$F^n=(F\circ\ell)^n\circ\ell_n$.  If $F$ and $F\circ\ell$ have a
common iterate, then so do $f$ and $g$.  By
Proposition~\ref{linearprop}, it remains only to consider the case
that $F=v^{-1} \circ \epsilon X^d \circ v$ and $\ell=v^{-1}\circ\delta
X\circ v$, where $v\in \Kbar[X]$ is linear and $\epsilon,\delta\in
\Kbar^*$.  Note that $d>1$.

By hypothesis, the set $\M$ of pairs $(m,n)\in\N\times\N$ satisfying
$f^m(x_0) = g^n(y_0)$ is infinite, and (from non-preperiodicity) its
projections onto each coordinate are injective.  Thus, for some
$s_1,s_2\in\N$, the set $\M$ contains infinitely many pairs
$(rm+s_1,rn+s_2)$ with $m,n \in \N$; since the projections are
injective, $\M$ contains pairs of this form in which $\min(m,n)$ is
arbitrarily large.  For any $m,n\in\N$ such that $(rm+s_1,rn+s_2)\in\M$,
we have
$F^m(x_1)=(F\circ\ell)^n(y_1)$, where $x_1 := f^{s_1}(x_0)$ and $y_1
:= g^{s_2}(y_0)$. Thus

\begin{equation*}
\begin{split}
   v^{-1}(\epsilon^{(d^m-1)/(d-1)} v(x_1)^{d^m}) & =
   F^m(x_1)\\
  & = (F\circ\ell)^n(y_1) \\
   & = v^{-1}((\epsilon\delta^d)^{(d^n-1)/(d-1)} v(y_1)^{d^n}),
\end{split}
\end{equation*}
so
\begin{equation}\label{molehill}
v(x_1)^{d^m} \epsilon^{(d^m-d^n)/(d-1)}  = \delta^{d(d^n-1)/(d-1)}
v(y_1)^{d^n}.
\end{equation}
We cannot have $v(x_1)=0$, since otherwise $x_1=f^{s_1}(x_0)$ is a fixed
point of $F=f^r$, contrary to our hypotheses.  Likewise $v(y_1)\ne 0$.
Now let $\epsilon_1, \delta_1\in\Kbar$ satisfy
$\epsilon_1^{d-1} = \epsilon$ and $\delta_1^{d-1} = \delta^d$, so
\eqref{molehill} implies
\begin{equation}
\label{molehill 2}
\delta_1 = v(x_1)^{-d^m} \cdot \epsilon_1^{d^n - d^m} \cdot
\delta_1^{d^n} \cdot v(y_1)^{d^n}.
\end{equation}

Since \eqref{molehill 2} holds for pairs $(m,n)$ with
$\min(m,n)$ arbitrarily large, there are infinitely many $k\in\N$
for which  $\delta_1$ is a $d^k$-th power in the number field
$K_0:=\Q(v(x_1),v(y_1),\epsilon_1,\delta_1)$.  Letting $\OO$ be the
ring of algebraic integers in $K_0$, it follows that the fractional
ideal of $\OO$ generated by $\delta_1$ is a $d^k$-th power for
infinitely many $k$; now unique factorization of fractional ideals
implies $\delta_1$ is in the unit group $U$ of $\OO$.  Moreover,
$\delta_1$ is a $d^k$-th power in $U$ for infinitely many $k$;
since $U$ is a finitely generated abelian group, $\delta_1$ must be
a root of unity whose order $N$ is coprime to $d$.  Thus
$N\mid (d^t-1)$ for some $t\in\N$.  Now $(F\circ\ell)^t=
v^{-1}\circ (\epsilon\delta^d)^{(d^t-1)/(d-1)} X^{d^t}\circ v$,
and since $\delta^d=\delta_1^{d-1}$ and $\delta_1^{d^t-1}=1$,
it follows that $(F\circ\ell)^t=F^t$, as desired.
\end{proof}


\section{Proof of Proposition~\ref{linearprop}}
\label{proofs}

In this section we complete the proof of Theorem~\ref{number fields},
by proving Proposition~\ref{linearprop}.
We consider two cases, depending on whether $F$ is gotten from a monomial
by composing with linears on both sides.  Our strategy is to show in both
cases that there are only finitely many linears $\hat\ell\in K[X]$ for
which there exists $n$ such that $(F\circ\ell)^n\circ\hat\ell = F^n$;
after this, we pick two values $n<N$ having the same $\hat\ell$, and
deduce that $F^{N-n}=(F\circ\ell)^{N-n}$.

\begin{lemma}
\label{key claim}
Let $K$ be a field of characteristic zero, and suppose $F\in K[X]$ has
the property that $u\circ F\circ v$ has at least two monomial terms
whenever $u,v\in K[X]$ are linear.
Then the equation $F\circ b =
a\circ F$ has only finitely many solutions in linear polynomials
$a,b\in K[X]$.
\end{lemma}

\begin{proof}[Proof of Lemma~\ref{key claim}.]
Our hypothesis implies $\deg(F)>1$.
Pick $\alpha,\beta\in K$ as in Lemma~\ref{Halblem}, and put
$\hat F:=\alpha+F(X+\beta)$; note that these choices depend only on $F$.
Then $\hat F\in X^r K[X^s]$ for some integers $r,s\ge 0$.  Our hypothesis
implies $s\ne 0$; now choose $s$ to be as large as possible.
By Lemma~\ref{Halblem}, if $F\circ b=a\circ F$ with $a,b\in K[X]$ linear,
then there is an $s\tth$ root of unity $\gamma\in K$ such that
$b=\beta+\gamma(X-\beta)$ and $a=-\alpha+\gamma^r(X+\alpha)$.
Since there are only finitely many possibilities for $\gamma$, there are
only finitely many possibilities for $a$ and $b$.
\end{proof}

\begin{remark}
Our proof shows that the number of solutions is less than $\deg(F)$
(in fact: the number of solutions is at most the size of the largest group of
roots of unity in $K$ of order less than $\deg(F)$).
\end{remark}

\begin{lemma}
\label{another one}
Let $K$ be a field of characteristic zero, let $u,v,\ell\in K[X]$ be linear,
and let $F=u\circ X^d\circ v$ where $d>1$.
The following are equivalent:
\begin{enumerate}
\item The equation
\begin{equation}\label{taketwo}
F\circ\ell\circ F\circ b = a\circ F\circ F
\end{equation}
has infinitely many solutions in linears $a,b\in K[X]$.
\item 
$F=v^{-1} \circ \epsilon X^d \circ v$ and
$\ell=v^{-1}\circ\delta X\circ v$ for some $\epsilon,\delta\in K^*$.
\end{enumerate}
\end{lemma}

\begin{proof}[Proof of Lemma~\ref{another one}.]
Pick any solution $(a,b)$ to (\ref{taketwo}).
By Lemma~\ref{uniq}, there is a linear $c\in K[X]$ such that
\[
\ell\circ F\circ b = c\circ F,
\]
which implies
\[
F\circ c = a\circ F.
\]
For any linears $e_1,e_2\in K[X]$ such that
$e_1\circ F\circ e_2 = F$, we have
\[
X^d = (u^{-1}\circ e_1\circ u) \circ X^d \circ (v\circ e_2\circ v^{-1}),\]
so $v\circ e_2\circ v^{-1} = \gamma X$ and $u^{-1}\circ e_1\circ u = X/\gamma^d$
for some $\gamma\in K^*$.
Thus, there exist $\gamma_1,\gamma_2\in K^*$ such that
\begin{align*}
b &= v^{-1}\circ\gamma_1 X\circ v \\
c^{-1}\circ\ell &= u\circ \frac{X}{\gamma_1^d}\circ u^{-1} \\
c &= v^{-1}\circ\gamma_2 X\circ v \\
a^{-1} &= u\circ \frac{X}{\gamma_2^d}\circ u^{-1}.
\end{align*}
We can eliminate $c$ from the second and third equations:
\[
u\circ \gamma_1^d X \circ u^{-1} = \ell^{-1}\circ c =
\ell^{-1}\circ v^{-1}\circ\gamma_2 X\circ v.
\]
Thus,
\[
\gamma_1^d X= (u^{-1}\circ \ell^{-1}\circ v^{-1})\circ \gamma_2 X\circ
(v\circ u).\]
Write $\alpha:=(v\circ u)(0)$.  Since $\gamma_1^d X$ fixes $0$, the linear
 polynomial
$h:=u^{-1}\circ\ell^{-1}\circ v^{-1}$ must map $\gamma_2\alpha$ to $0$.
Since $\alpha$ and $h$ do not depend on $a$ and $b$, it follows that if
$\alpha\ne 0$ then $\gamma_2$ (and thus $\gamma_1^d$) does not depend on
$a$ and $b$, so there are
only finitely many possibilities for $a$ and $b$.
Now assume $\alpha=0$, so $0$ is fixed by both $v\circ u$ and
$u^{-1}\circ\ell^{-1}\circ v^{-1}$, whence these two linears have the form
$\epsilon X$ and $\hat\delta X$ with $\epsilon,\hat\delta\in K^*$.  Then
 $u=v^{-1}\circ
\epsilon X$ and $\ell^{-1}=v^{-1}\circ \epsilon\hat\delta X\circ v$,
so $F=v^{-1}\circ\epsilon X^d\circ v$ and (with $\delta=1/(\epsilon\hat\delta)$)
we have $\ell = v^{-1}\circ \delta X\circ v$.

It remains only to show that, when $F=v^{-1}\circ \epsilon X^d\circ v$ and
$\ell = v^{-1}\circ \delta x\circ v$, the number of solutions of (\ref{taketwo})
is infinite.  To this end, pick any $\iota\in K^*$, and note that
$b=v^{-1}\circ\iota X\circ v$ and $a=v^{-1}\circ \delta^d\iota^{d^2}X\circ v$
satisfy (\ref{taketwo}).
\end{proof}

\begin{remark}
This proof shows that, when the number of solutions to (\ref{taketwo}) is finite,
this number is at most $d$ (in fact: at most
the number of $d\tth$ roots of unity in $K$).
\end{remark}

\begin{proof}[Proof of Proposition~\ref{linearprop}.]
We have
\begin{equation}\label{useful}
(F\circ\ell)^n\circ\ell_n = F^n\end{equation} for every $n$ in some
infinite subset $\M$ of $\N$.  For $n\in \M$, we apply Lemma~\ref{uniq} to
(\ref{useful}) with
$B=F\circ\ell\circ\ell_n$ and $D=F$, to conclude that there is a linear
$u_n\in K[X]$ such that 
\[
F\circ \ell\circ\ell_n=u_n \circ F.
\]
By Lemma~\ref{key claim}, if $F$ is not gotten from a monomial by composing
with linears on both sides, then $\{\ell_n:n\in \M\}$ is finite.

Next, for $n\in \M$ with $n>1$, apply Lemma~\ref{uniq} to (\ref{useful}) with
$B=(F\circ\ell)^2\circ\ell_n$ and $D=F^2$, to conclude that there is a
linear $v_n\in K[X]$ such that
\[
(F\circ\ell)^2\circ\ell_n = v_n\circ F^2.
\]
By Lemma~\ref{another one}, if $F$ is gotten from a monomial by composing
with linears on both sides, then either $\{\ell_n:n\in \M\}$ is finite
or conclusion $(2)$ of Proposition~\ref{linearprop} holds.

Thus, whenever $(2)$ of Proposition~\ref{linearprop} does not hold, the set
$\{\ell_n:n\in \M\}$ is finite, so there exist $n,N\in \M$ such that
$\ell_n=\ell_N$ and $n<N$.  Then
\begin{align*}
F^{N-n} \circ F^n &= F^N \\
&= (F\circ\ell)^N \circ \ell_n \\
&= (F\circ\ell)^{N-n} \circ (F\circ\ell)^n \circ \ell_n \\
&= (F\circ\ell)^{N-n} \circ F^n,
\end{align*}
so $F^{N-n} = (F\circ\ell)^{N-n}$, as desired.
\end{proof}


\section{Some reductions}
\label{trivial action}

In this section we show that it suffices to prove Theorems~\ref{main result}
and \ref{complex2}
in case $K$ is a finitely generated extension of $\Q$.  Moreover, for any
such $K$, it suffices to prove these results in case
$\deg(f)=\deg(g)>1$ and the line is the diagonal, $X=Y$.

We begin with the first reduction.
For fixed $K,f,g,x_0,y_0,L$, only finitely
many elements of $K$ occur as coefficients of $f$ or $g$, as values $x_0$
or $y_0$, or in the defining equation for $L$.  Let $K_0$
be the extension of $\Q$ generated by these finitely many elements.
Then Theorem~\ref{main result} holds for $(K,f,g,x_0,y_0,L)$ if it holds for
$(K_0,f,g,x_0,y_0,L)$, and likewise for Theorem~\ref{complex2}.

We next show that we need only consider the case that the line is the diagonal.

\begin{lemma}\label{linear equals diagonal}
If Theorem~\ref{main result} is true for the line $X=Y$, then it is
true for every line.
\end{lemma}

\begin{proof}
If $L$ has the form $X=\alpha$ then the theorem is obvious:
if there are infinitely many $n$ such that $f^n(x_0) = \alpha$, then
$\alpha$ is periodic point for $f$, so $X = \alpha$ is a periodic line
for $(f,g)$.  Likewise the result is
clear if $L$ has the form $Y=\beta$, so we may assume $L$ is
$X = \ell(Y)$ with $\ell\in K[Y]$ of degree one.
Suppose $\{(f^n(x_0),g^n(y_0)):n\in\N\}$ has infinite intersection
with $L$.  If $f^n(x_0) = \ell(g^n(y_0))$ then
$f^n(x_0) = (\ell\circ g\circ \ell^{-1})^n(\ell(y_0))$.
Thus, assuming Theorem~\ref{main result} for the line $X=Y$,
we conclude that $X=Y$ is periodic under the action of $(f,\ell\circ g\circ
\ell^{-1})$; it follows that $X=\ell(Y)$ is periodic under the
$(f,g)$-action.
\end{proof}

The analogous result for Theorem~\ref{complex2} follows from a similar
argument.

\begin{lemma}\label{linear equals diagonal2}
If Theorem~\ref{complex2} is true in case $\alpha=1$ and $\beta=0$,
then it is true for arbitrary $\alpha$ and $\beta$.
\end{lemma}

We now prove Theorem~\ref{main result} in case $f$ and $g$ are linear
polynomials.  As noted in the introduction, Theorem~\ref{complex2}
fails in this case, so our proof must necessarily distinguish between
the equations $f^n(x_0)=g^n(y_0)$ and $f^m(x_0)=g^n(y_0)$.

\begin{prop}
\label{linear polynomials}
Theorem~\ref{main result} holds if $\deg(f) = \deg(g) =1$ and $L$ is the
diagonal.
\end{prop}

\begin{proof}
  Suppose the hypotheses of Theorem~\ref{main result} hold.  As above,
  we may assume $K\subseteq\C$.  By replacing $x_0$ and $y_0$ with
  $f^{n_0}(x_0)$ and $g^{n_0}(y_0)$ (for some $n_0\in\N$), we may
  assume $x_0 = y_0$.  Let $f(X) = \alpha X +\beta$ and $g(X) = \gamma
  X+\delta$.  Note that $\alpha$ cannot be a root of unity different
  from $1$, for otherwise some iterate of $f$ would be the identity
  map, contradicting infinitude of $\{f^n(x_0):n\in\N\}$.  Likewise,
  $\gamma$ is not a root of unity different from $1$.  We consider two
  cases:

\emph{Case 1.} Neither $\alpha$ nor $\gamma$ equals $1$.

For $n\in \N$, we have $f^n(x_0) = \alpha^n\hat{x}_0 -
\frac{\beta}{\alpha-1}$ and $g^n(x_0) = \gamma^n \hat{y}_0 -
\frac{\delta}{\gamma-1}$, where
$\hat{x}_0:=x_0+\frac{\beta}{\alpha-1}$ and $\hat{y}_0:=x_0 +
\frac{\delta}{\gamma-1}$.  Since $x_0$ is not preperiodic for $f$ or
$g$, both $\hat{x}_0$ and $\hat{y}_0$ are nonzero.  By the hypothesis
of Theorem~\ref{main result}, there are infinitely many $n\in\N$ such
that $\alpha^n \hat{x}_0 - \gamma^n \hat{y}_0 = \hat{x}_0 -
\hat{y}_0$.  If $\hat{x}_0 \neq \hat{y}_0 $, we may divide through
and obtain infinitely many $n$ such that
$$
\hat{a} \alpha^n + \hat{b} \gamma^n = 1$$
for some constants $\hat{a}$ and $\hat{b}$.
As noted by Lang \cite[p.~28]{Lang_int}, this is impossible (as
can be seen by passing to a curve $\hat{a} \alpha^i t^3 +
\hat{b} \gamma^i u^3 = 1$, with $0 \leq i \leq 2$, and using Siegel's
theorem on integral points).
Hence, we must have $\hat{x}_0 = \hat{y}_0$, so there are infinitely
many $n\in\N$ for which $\alpha^n = \gamma^n$, and $f^n=g^n$ for each
such $n$.

\emph{Case 2.} Either $\alpha$ or $\gamma$ equals $1$.

Without loss of generality, we may assume $\alpha=1$. If also $\gamma=1$,
then since $f^n(x_0) = g^n(x_0)$ for
some $n\in\N$, we must have $\beta=\delta$, so $f=g$ as desired. 
Now assume $\gamma\ne 1$. Then
$g^n(x_0) = \gamma^n\left( x_0+\frac{\delta}{\gamma-1}\right) - \frac{\delta}
{\gamma-1}$.
Since $\{g^n(x_0):n\in\N\}$ is infinite, we must have
$x_0\ne -\delta/(\gamma-1)$.  By hypothesis, there are infinitely many
$n\in\N$ such that
\begin{equation}
\label{order of growth}
x_0+n\beta = \gamma^n\left( x_0+\frac{\delta}{\gamma-1}\right) -
 \frac{\delta}{\gamma-1}.
\end{equation}
This is not possible if $|\gamma|>1$, since then the absolute value of the right
side exceeds that of the left side for sufficiently large $n$.  Thus
$|\gamma|\le 1$, so the right side is bounded independently of $n$, whence
also $x_0+n\beta$ is bounded.  This implies $\beta=0$, so $f$ is the identity map,
contradicting the hypothesis that $\{f^n(x_0):n\in\N\}$ is infinite.
\end{proof}

\begin{remark}
We note that the argument used in Case 2 above does not generalize
to the setting of Theorem~\ref{complex},
since we used in a crucial way that we have only one variable $n$ in
\eqref{order of growth}, so the orders of growth of the two sides
of \eqref{order of growth} are different. 
In fact, the conclusion of Theorem~\ref{complex} is not generally
true if $f$ is a monic linear polynomial. For example, let
$f(X) = X + 1$ and let $g(X)$ be any nonconstant polynomial with
positive integer coefficients.  Then
for any positive integers $x_0$ and $y_0$ such that $g(y_0)>y_0$, the
intersection $\OO_f(x_0)\cap\OO_g(y_0)$ is infinite, since
$\OO_f(x_0)$ contains every sufficiently large integer.   On the other
hand, the argument from Case 1 generalizes at once to show that
the conclusion of Theorem~\ref{complex} holds when $f$ and $g$
are non-monic linear polynomials.
\end{remark}

The remainder of this section is devoted to proving
Theorem~\ref{main result} in case $K$ is a number field and
$\deg(f)\ne\deg(g)$.  We
recall some standard terminology: a \emph{global field} is either a
number field or a function field of transcendence degree $1$ over
another field.  Any global field $E$ comes equipped with a set $M_E$
of normalized absolute values $||\cdot||_v$ which satisfy a product
formula\footnote{A `normalized absolute value' is a power of an
  absolute value, but might not be an absolute value itself since it
  might fail the triangle inequality.}:

\[
\prod_{v\in M_E} ||x||_v = 1\quad\text{ for every $x\in E^*$}.
\]

If $E$ is a global field, the logarithmic Weil height of $x\in\overline{E}$
is defined as
$$h(x) = \frac{1}{[E(x):E]}\cdot \sum_{v\in
  M_E}\sum_{\substack{w|v\\ w\in M_{E(x)}}}
\log\max\{||x||_w , 1\}.$$

We will use the following easy consequence of these definitions
 (cf.\ \cite[p.\ 77]{Lang_diophantine}).

\begin{lemma}
\label{easy fact}
Let $E$ be a global field, and let $\ell\in E[X]$ be a linear
polynomial. Then there exists $c_{\ell} > 0 $ such that $| h(\ell(x))
-  h(x)| \le c_{\ell}$\/ for all $x\in\overline{E}$.
\end{lemma}

\begin{defi}
\label{definition canonical height}
Let $E$ be a global field, let $f\in E[X]$ with $\deg(f)>1$,
and let $z\in \overline{E}$. The
canonical height $\hh_f(z)$ of $z$ with respect to the morphism
$f:\bP^1\lra\bP^1$ is
$$\widehat{h}_f(z) = \lim_{k\rightarrow\infty} \frac{h(f^k(z))}{\deg(f)^k}.$$
\end{defi}

This definition is due to Call and Silverman, who proved the existence
of the above limit in \cite[Thm.\ $1.1$]{Call-Silverman} by using
boundedness of $|h(f(x)) - (\deg f) h(x)|$ and a
telescoping series argument due to Tate.  We will use the
following properties of the canonical height.
\begin{lemma}\label{can}
Let $E$ be a global field, let $f\in E[X]$ be a polynomial of degree
greater than $1$, and let $z\in {\overline{E}}$. Then
\begin{itemize} 
\item[(a)] for each $k\in\N$, we have $\hh_f(f^k(z)) = \deg(f)^k \cdot
  \hh_f(z)$;
\item[(b)] $|h(z) - \hh_f(z)|$ is uniformly bounded independently
of $z\in{\overline{E}}$;
\item[(c)] if $E$ is a number field, $z$ is preperiodic if and only if
 $\hh_f(z) = 0$.
\end{itemize}
\end{lemma}

\begin{proof}
Part (a) is clear; for (b) and (c)
see \cite[Thm.\ 1.1 and Cor.\ 1.1.1]{Call-Silverman}.
\end{proof}
Part (c) of Lemma~\ref{can} is not true if $E$ is a function
field with constant field $E_0$, since $\hh_f(z)=0$ whenever $z\in E_0$
and $f\in E_0[X]$.  But these are essentially the only counterexamples
in the function field case (cf.\ Lemma~\ref{Benedetto}).

\begin{lemma}
\label{same degree}
Let $K$ be a number field, let $f,g \in K[X]$ and let $x_0,y_0\in K$.
If $\OO_{(f,g)}((x_0,y_0))$ has infinitely many points on the diagonal,
then $\deg(f)=\deg(g)>0$.
\end{lemma}

\begin{proof}
  The hypothesis implies $x_0$ (resp., $y_0$) is not preperiodic
for $f$ (resp., $g$).  Thus $f$ and $g$ are nonconstant.  Suppose
$\deg(f)>\deg(g)$.
  
Since $\hh_f(x_0) > 0$ (by Lemma~\ref{can}), there exists
  $\delta > 0$ such that every sufficiently large $k$ satisfies
$$ h(f^k(x_0)) > (\deg f)^k \delta.$$
If $\deg g=1$, by Lemma~\ref{easy fact} there exists $c_g>0$ such
that
$$ h(g^k(y_0)) \le k c_g + h(y_0)$$
for every $k$, and for sufficiently large $k$ we have
$(\deg f)^k \delta > k c_g + h(y_0)$.  
If $\deg g > 1$, there exists $\epsilon > 0$ such that every
$k$ satisfies
$$
h(g^k(y_0)) < (\deg g)^k \epsilon,$$
and since $\deg f>\deg g$ we have $(\deg f)^k \delta > (\deg g)^k\epsilon$
for $k$ sufficiently large.  Hence, in either case, for $k$ sufficiently
large we have $h(f^k(x_0)) > h(g^k(y_0))$ and thus $f^k(x_0)\ne g^k(y_0)$.
\end{proof}

\begin{remark}
\label{still linear}
This proof does not work for function fields, since it relies on
Lemma~\ref{can} (c).  However, one can use a different argument
to show that Lemma~\ref{same degree} is valid for any field $K$ (of
any characteristic).
In characteristic zero, this is a consequence of Theorem~\ref{main result}.
One can prove this for general $K$ using arguments similar to those in
this paper; the key intermediate result is that,
for any $f\in K[X]$ with $\deg(f)>1$, and any $z\in K$ non-preperiodic
for $f$, there is an absolute value $v$ of $K$ such that
$\lim_{n\to\infty}|f^n(z)|_v=+\infty$.
\end{remark}


\section{The function field case}
\label{over C}
In this section we prove Theorems~\ref{main result} and \ref{complex2}.
Our strategy is to `specialize' every transcendental generator of $K$
to an element of a number field, and then deduce these results from the
number field version proved previously (Theorem~\ref{number fields}).
We begin by proving that Theorem~\ref{main result} follows from the existence
of a suitable specialization homomorphism.

\begin{proof}[Proof of Theorem~\ref{main result}, assuming existence of a
 suitable specialization.]

From the results of the previous section, it suffices to prove
Theorem~\ref{main result} in case $K$ is a finitely generated extension
of $\Q$, the line $L$ is the diagonal, and $\deg(f)\ge 2$.  We will prove
Theorem~\ref{main result} by
induction on the transcendence degree of $K/\Q$.  The base case is
Theorem~\ref{number fields} and Lemma~\ref{same degree}.  For the inductive
step, let $E$ be a subfield of $K$ such that $\trdeg(K/E)=1$ and
$E/\Q$ is finitely generated.  Suppose in
addition that the diagonal is not periodic under the $(f,g)$ action (i.e.,
there is no $k\in\N$ for which $f^k=g^k$), and that the set
$\{(f^n(x_0),g^n(y_0)):n\in\N\}$ has infinite intersection with the diagonal.
Assume there is a subring $R$ of $K$, a finite extension $E'$
of $E$, and a homomorphism
$\alpha:R\to E'$, such that
\begin{enumerate}
\item $R$ contains $x_0$, $y_0$, and every coefficient of
$f$ and $g$, but the leading coefficients of $f$ and $g$ have nonzero image
under $\alpha$;
\item $f_\alpha^k\ne g_\alpha^k$ for each $k\in\N$;
\item $x_{0,\alpha}$ is not preperiodic for $f_{\alpha}$.
\end{enumerate}
(Here $f_\alpha$, $g_\alpha$, and $x_{0,\alpha}$ denote the images of
$f$, $g$, and $x_0$, respectively, under the homomorphism $\alpha$.)  

Properties $(1)$ and $(3)$ show that
 $\{(f_{\alpha}^n(x_{0,\alpha}),g_{\alpha}^n(y_{0,\alpha})):n\in\N\}$
 has infinite intersection with the diagonal. The inductive hypothesis
 implies $f_{\alpha}^k = g_{\alpha}^k$ for some $k\in\N$, which
 contradicts property $(2)$.
 Theorem~\ref{main result} follows.
\end{proof}

The proof of Theorem~\ref{complex2} is nearly identical to the proof of
Theorem~\ref{main result}, the only difference being that we replace
the set $\{(f^n(x_0),g^n(y_0)):n\in\N\}$ with
$\OO_f(x_0)\times\OO_g(y_0)$.

To explain why there exists an $\alpha$ as in the proof of
Theorem~\ref{main result}, we recall the usual setup for
specialization.  By replacing $E$ with a finite extension of $E$, we
may assume $E$ is algebraically closed in $K$.  Let $C$ be a smooth
projective curve over $E$ whose function field is $K$,
and let $\pi:\bP^1_C\to C$ be the natural fibration.  Any
$z\in\bP^1_K$ gives rise to a section $Z:C\to\bP^1$ of $\pi$, and for
$\alpha\in C(\overline{E})$, we let $z_\alpha:=Z(\alpha)$,
and let $E(\alpha)$ be the residue field of $K$ at the valuation
corresponding to $\alpha$.  In the notation of the previous paragraph,
$R$ is the valuation ring for this valuation, $E'$ is $E(\alpha)$, and
the homomorphism $R\to E'$ is $z\mapsto z_\alpha$.  The polynomial
$f\in K[X]$ extends to a rational map (of $E$-varieties) from
$\bP^1_C$ to itself, whose generic fiber is $f$, and whose fiber above
any $\alpha\in C$ is $f_{\alpha}$.  Note that $f_{\alpha}$ is a
morphism of degree $\deg(f)$ from the fiber
$(\bP^1_C)_\alpha=\bP^1_{E(\alpha)}$ to itself whenever the
coefficients of $f$ have no poles or zeros at $\alpha$; hence it is a
morphism on $\bP^1_{E(\alpha)}$ of degree $\deg(f)$ at all but finitely many
$\alpha$ (we call these $\alpha$ places of \emph{good reduction} for $f$).

Intuitively, we will show that most choices of $\alpha$ satisfy conditions
(2) and (3) above (obviously all but finitely many $\alpha$ satisfy (1)).

We will first prove the following result about specializations of
polynomials. 

\begin{prop}\label{Ritt case}
For each $r>0$, there are at most finitely many $\alpha\in C(\overline{E})$
such that $[E(\alpha):E]\le r$ and $f_{\alpha}^k=g_{\alpha}^k$ for some $k\in\N$.
\end{prop}

Next, letting $h_C$ be the logarithmic Weil height on $C$ associated to a
fixed degree-one ample divisor, we will prove the following dynamical
analogue of Silverman's specialization result for abelian varieties
\cite[Thm.\ $C$]{Silverman_specialization}.

\begin{prop}
\label{Silverman case}
There exists $c>0$ such that, for $\alpha\in C(\overline{E})$ with
$h_C(\alpha)>c$, the point $x_{0,\alpha}$ is not preperiodic for
$f_{\alpha}$.
\end{prop}

We now show that these two results imply the existence of $\alpha$
satisfying (1)--(3), which in turn implies Theorems \ref{main result}
and \ref{complex2}.  Let
$\phi:C\to \bP^1_E$ be any nonconstant rational function, and let
$r=\deg(\phi)$.  By \cite[Prop.\ 4.1.7]{Lang_diophantine}, there
are positive constants $c_1$ and $c_2$ such that for all
$P \in \bP^1(\overline{E})$, the
preimage $\alpha=\phi^{-1}(P)$ satisfies $h_C(\alpha) \geq
c_1 h(P)+ c_2$.  Since there are infinitely many $P \in
\bP^1(E)$ such that $h(P) > (c - c_2)/c_1$, we thus obtain infinitely
many $\alpha \in C(\overline{E})$ such that $h_C(\alpha) > c$ and
$[E(\alpha):E]\le r$.  Hence, Propositions~\ref{Ritt case} and
\ref{Silverman case} imply there are infinitely many $\alpha$
satisfying (2) and (3), and all but finitely many of these satisfy (1)
as well.

\subsection{Polynomials with a common iterate}
\label{Ritt proof}

In this section we prove Proposition~\ref{Ritt case}.

Our proof relies on a classical result of Ritt \cite[p.~356]{Rittit}
describing the pairs of complex polynomials having a common
iterate, i.e., $F^n=G^m$ for some $n,m\in\N$.  We only need this
for $n=m$, in which case Ritt's result is as follows.

\begin{prop}
\label{Rittprop}
Let $F,G\in\C[X]$ with $d:=\deg(F)>1$.  For $n\in\N$, we have
$F^n=G^n$ if and only if 
$F(x)=-\beta+\gamma H(x+\beta)$ and
  $G(x)=-\beta+H(x+\beta)$ for some $\gamma\in \C^*$,
  $\beta\in\C$ and $H\in x^r\C[x^s]$
  (with $r,s\ge 0$) such that $\gamma^s = 1$
  and $\gamma^{(d^n-1)/(d-1)}=1$.  
\end{prop}

\begin{cor}\label{Ritti}
Let $K$ be a field of characteristic zero, and let $N_K$ be the number
of roots of unity in $K$.  Let $F,G\in K[X]$ satisfy $\deg(F)=d>1$ and
$F^k=G^k$ for some $k\in\N$.  Then $F^n=G^n$ for some $n$ with $1\le n\le N_K$.
\end{cor}

\begin{proof}[Proof of Corollary~\ref{Ritti}.]
Let $K_0$ be the subfield of $K$ generated by the coefficients of $F$
and $G$.  Then $K_0$ is a finitely generated extension of $\Q$, so $K_0$
is isomorphic to a subfield of $\C$.  After identifying $K_0$ with its
image in $\C$, Proposition~\ref{Rittprop} implies that
$F=-\beta+\gamma H(x+\beta)$ and $G=-\beta+H(x+\beta)$ for some
$\gamma\in\C^*$, $\beta\in\C$, and $H\in x^r\C[x^s]$
(with $r,s\ge 0$) such that $\gamma^s=1$.  Moreover, for $n\in\N$ we have
$F^n=G^n$ if and only if $\gamma^{(d^n-1)/(d-1)}=1$.
Since $\gamma$ is the ratio of the leading coefficients of $F$ and $G$,
we see that $\gamma\in K_0^*$.  Since $\gamma^{(d^k-1)/(d-1)}=1$, the
multiplicative order $m$ of $\gamma$ is coprime to $d$.
Note that $m\le N_K$.

Let $p$ be a prime factor of $m$, and let $p^t$ be the maximal power of $p$
dividing $m$.  If $p\nmid (d-1)$ then let $q_p$ be the
order of $d$ in $(\Z/p^t)^*$; otherwise, put $q_p=p^t$.
Then $n:=\prod q_p$ satisfies $n\le m$ and $m\mid (d^n-1)/(d-1)$, whence
$n\le N_K$ and $F^n=G^n$.
\end{proof}

\begin{proof}[Proof of Proposition~\ref{Ritt case}.]
Pick a point $\ab$ on $C$ such that $[E(\ab):E]\le r$ and
$f_{\ab}^k = g_{\ab}^k$ for some $k\in \N$.  
Let $N_{\ab}$ be the number of roots of unity in $E(\ab)$.
By Corollary~\ref{Ritti}, the least $n\in\N$ with
$f_{\ab}^n = g_{\ab}^n$ satisfies $n\le N_{\ab}$.
Now, $N_{\ab}$ is bounded in terms of the degree
$[E(\ab) \cap {\overline \Q}: \Q]$, which is at most
$r\cdot [E\cap\overline\Q:\Q]$; since $E$ is finitely
generated, the latter number is finite, so there is a
finite bound on $n$ which depends only on $E$ and $r$
(and not on $\ab$).

For any fixed $n\in\N$, we have $f^n \ne g^n$, so $\deg(f_{\ab}^n-g_{\ab}^n)=
\deg(f^n-g^n)\ge 0$ for all but finitely many $\ab\in C$.
The result follows.
\end{proof}

\subsection{Specialization of non-preperiodic points}
\label{Silverman proof}

In this section we prove Proposition~\ref{Silverman case}.

First note that $E$ is a global field.
The key ingredient in our proof is the following result of Call
and Silverman \cite[Thm.\ 4.1]{Call-Silverman}, which relates $h_C$
to the canonical heights $\hh_f:\overline{K}\to\R_{\ge 0}$
 and $\hh_{f_{\alpha}}:\overline{E}\to \R_{\ge 0}$
of $f$ and $f_{\alpha}$
(cf.\ Definition~\ref{definition canonical height}).

\begin{lemma}\label{CS prop}

For each $z\in K$ we have
\begin{equation}\label{CS}
\lim_{h_C(\alpha) \to \infty} \frac{\hh_{f_\alpha}(z_\alpha)}{h_C(\alpha)}
= \hh_{f}(z).
\end{equation}
\end{lemma}

We will also use a
result about canonical heights of non-preperiodic points for
polynomials that are not \emph{isotrivial}.
\begin{defi}
\label{isotrivial}
We say a polynomial $f \in K[X]$ is isotrivial if there
exists a finite extension $K'$ of $K$ and a linear $\ell
\in K'[X]$ such that $\ell^{-1}\circ f \circ \ell \in \overline{E}[X]$.
\end{defi}
Benedetto proved that a non-isotrivial polynomial can only have
canonical height equal to $0$ at its preperiodic points \cite[Thm.\ $B$]{Bene}:

\begin{lemma}
\label{Benedetto}
  Let $f\in K[X]$ with $\deg(f)\ge 2$, and let
$z\in \overline{K}$. If $f$ is not isotrivial, then
$ \hh_{f}(z) = 0$ if and only if $z$ is preperiodic for~$f$.
\end{lemma}

We need one more preliminary result.
\begin{lemma}
\label{height 0 for isotrivial}
Let $f\in K[X]$ be isotrivial with $\deg(f)\ge 2$, and let $\ell$ be as
in Definition~\ref{isotrivial}.  If $z\in\overline{K}$ satisfies
$\hh_f(z) = 0$, then $\ell^{-1}(z)\in\overline{E}$.
\end{lemma}

\begin{proof}
Put $F:=\ell^{-1}\circ f \circ \ell \in K'[X]$, so
$F^n(\ell^{-1}(z)) = \ell^{-1} (f^n(z))$.
Since $\hh_f(z)=0$, Lemma~\ref{easy fact} implies that
$\hh_F(\ell^{-1}(z)) = 0$.  For any $v\in M_{K'(z)}$,
we know that every nonzero coefficient $\gamma$ of $F$ satisfies
$||\gamma||_v=1$ (since $\gamma\in \overline{E}$).  Since $v$ is
nonarchimedean, if $y\in K'(z)$ satisfies $||y||_v>1$ then
$\log ||F^n(y)||_v=\deg(F)^n \log ||y||_v$, so $\hh_F(y)>0$.
Thus $||\ell^{-1}(z)||_v\le 1$ for every $v\in M_{K'(z)}$, so
$\ell^{-1}(z)\in\overline{E}$.
\end{proof}

\begin{proof}[Proof of Proposition~\ref{Silverman case}.]
Put $z=x_0$.
   If $ \hh_{f}(z) > 0$ then, by Lemma~\ref{CS prop}, there exists $c>0$ such
   that every $\alpha \in C({\overline E})$ with $h_C(\alpha) > c$ satisfies
  $$
  \frac{\hh_{f_\ab}(z_\ab)}{h_C(\ab)} > 0.$$
Then $\hh_{f_\ab}(z_\ab) > 0$, so part (a) of Lemma~\ref{can} implies
$z_\ab$ is not preperiodic for $f_\ab$.

If $f$ is not isotrivial, Lemma~\ref{Benedetto} implies $\hh_f(z)>0$,
so the proof is complete.  It remains only to consider the case that $f$ is
isotrivial and $\hh_f(z) = 0$.

Pick a finite extension $K'$ of $K$ and a linear $\ell\in K'[X]$
such that $g := \ell^{-1} \circ f\circ \ell$ is in $\overline{E}[X]$, and put
$E':=\overline{E}\cap K'$.
Lemma~\ref{height 0 for isotrivial} implies
$w:=\ell^{-1}(z)$ is in $E'$.  Moreover, since
$\ell^{-1}\circ f^n(z) = g^n(w)$ and $z$ is not preperiodic for $f$,
we see that $w$ is not preperiodic for $g$. Because $g\in E'[X]$ and $w\in E'$,
then for all places $\alpha'$ of $K'$, the reductions of $g$ and $w$ at $\alpha'$
equal $g$, and respectively $w$ (because $E'$ embeds naturally into the residue
field at $\alpha'$). Hence, for all but finitely many $\alpha'$ (we only need to
exclude the places where $\ell$ does not have good reduction), if $\alpha$ is
the place of $K$ lying below $\alpha'$, then $z_{\ab}$ is not preperiodic for
$f_{\ab}$.

\end{proof}


\section{Further conjectures}

We suspect that Theorem~\ref{complex2} remains true without the
hypothesis that $\deg(f) = \deg(g)$.  It might be possible to prove
this by methods similar to those in this paper; however, this seems to
require substantial effort, since the results of Bilu-Tichy and Ritt
which we used became much simpler in our case $\deg(f)=\deg(g)$.

It would be interesting to study Conjecture~\ref{dynamical M-L} for
other curves in the plane.  In particular, it may be possible to treat
curves of the form $F(X)=G(Y)$ (with $F,G$ polynomials) by methods
similar to ours.



\end{document}